\documentclass[pdftex,a4paper]{amsart}

\usepackage{amsfonts, amsthm, amsmath, amssymb}
\usepackage{mathrsfs}
\usepackage[english]{babel}
\usepackage[pdftex,pdfauthor={Dario Trevisan}]{hyperref}
\usepackage[pdftex]{color,graphicx}

\hypersetup{colorlinks=true}

\newtheorem{theorem}{Theorem}[section]
\newtheorem{definition}[theorem]{Definition}

\newtheorem{lemma}[theorem]{Lemma}
\newtheorem{proposition}[theorem]{Proposition}
\newtheorem{corollary}[theorem]{Corollary}

\theoremstyle{remark}
\newtheorem{remark}[theorem]{Remark}
\newtheorem{example}[theorem]{Example}

\numberwithin{equation}{section}
\numberwithin{figure}{section}

\newcommand{\bra}[1]{\left( #1 \right)}

\newcommand{\cur}[1]{\left\{ #1 \right\}}

\newcommand{\abs}[1]{\left| #1 \right|}
\newcommand{\nor}[1]{\left\| #1 \right\|}
\newcommand{\res}{\llcorner}
\newcommand{\R}{\mathbb{R}}
\newcommand{\N}{\mathbb{N}}

\renewcommand{\d}{\mathsf{d}}
\newcommand{\DD}{D}

\newcommand{\loc}{\operatorname{loc}}

\renewcommand{\div}{\operatorname{div}}
\newcommand{\eps}{\varepsilon}

\title[Nonsmooth vector fields and commutation of their flows]{Lie brackets of nonsmooth vector fields and commutation of their flows}

\begin{document}
	\author{Chiara Rigoni}
\address{Chiara Rigoni, Institut f\"ur Angewandte Mathematik, Universit\"at Bonn, Endenicher Allee 60\\ D-53115, Bonn
}
\email{rigoni@iam.uni-bonn.de}

	\author{Eugene Stepanov}
\address{Eugene Stepanov, St.Petersburg Branch of the Steklov Mathematical Institute of the Russian Academy of Sciences,
	Fontanka 27,
	191023 St.Petersburg, Russia
	\and
	HSE University, Moscow
	\and ITMO University
}
\email{stepanov.eugene@gmail.com}

\author{Dario Trevisan}
\address{Dario Trevisan, Dipartimento di Matematica, Universit\`a di Pisa \\
	Largo Bruno Pontecorvo 5 \\ I-56127, Pisa}
\email{dario.trevisan@unipi.it}

\thanks{	The first author   gratefully acknowledges support by the European Union through the
ERC-AdG 694405 RicciBounds. 	The work of the second author has been also partially financed by
	the Program of the Presidium of the Russian
	Academy of Sciences \#01 'Fundamental Mathematics and its Applications'
	under grant PRAS-18-01 and
	RFBR grant \#20-01-00630A. The third author was partially supported by GNAMPA-INdAM 2019 project "Proprietà analitiche e geometriche di campi aleatori" and University of Pisa, Project PRA 2018-49.
}
\date{\today}

\maketitle

\begin{abstract}
It is well-known that the flows generated by two smooth vector fields commute, if the Lie bracket of these vector fields vanishes.
This assertion is known to extend to Lipschitz continuous vector fields, up to interpreting the vanishing of their Lie bracket 
in the sense of almost everywhere equality. We show that this cannot be extended to general a.e. differentiable vector fields admitting a.e. unique flows. We show however that the extension holds when one field is Lipschitz continuous and the other one is merely Sobolev regular (but admitting a regular Lagrangian flow).
\end{abstract}

\section{Introduction}
One of the well-known basic facts of differential geometry, ultimately leading to the Frobenius theorem on integral manifolds, is that the flows of two smooth vector fields
$V^1$, $V^2$ over a smooth finite dimensional manifold $M$ commute, if and only if their Lie bracket $[V^1, V^2]$ vanishes.
In a more formal way, denoting the flow of $V^i$ by $F_t^i(x):= y(t)$, where $y(\cdot)$ is the integral curve of the differential equation
\begin{equation}\label{eq_FrobCauchy1}
\dot{y}=V^i(y), \quad y(0)=x,
\end{equation}
one has
\begin{equation}\label{eq_FrobCommut1}
F_t^1(F_s^2(x))= F_s^2(F_t^1(x))\quad\mbox{for all $x\in M$, whenever } [V^1, V^2] = 0, 
\end{equation}
once $t\in \R$ and $s  \in \R$ are such that the respective expressions are defined. What happens with this statement for
possibly nonsmooth vector fields $V^i$, even when the underlying manifold $M$ itself remains smooth (in the sequel
we will for simplicity consider just the case of $M=\R^d$, a finite-dimensional Euclidean space)?

Of course, to be meaningful, the question posed has to be restricted to the cases when all the objects present in~\eqref{eq_FrobCommut1} are well-defined. This is however nontrivial already in the case when $V^i$ are Lipschitz. Namely, in this case the flows of $V^i$ are still defined because the Cauchy problem~\eqref{eq_FrobCauchy1} has a unique solution for \textit{every} initial datum $x\in \R^d$. However, there is a problem with the meaning of 
the Lie bracket  $[V^1, V^2]$; in fact, the latter is defined for smooth vector fields by the formula
\begin{equation}\label{eq_FrobCommutDef1}
 [V^1, V^2](x) := DV^2(x) V^1(x)- DV^1(x) V^2(x),
\end{equation}
where $DV^i$ are Jacobi matrices of $V^i$, $i=1,2$, which is meaningless when $V^i$ are only Lipschitz, because then their
derivatives are not necessarily defined for \textit{all} $x\in \R^d$ but only for \textit{almost all} with respect to the Lebesgue measure.
The natural question would be then: does~\eqref{eq_FrobCommut1} hold for a.e. - rather than for all, $x\in \R^d$ - once
the vanishing of the commutator $[V^1, V^2] = 0$ is also understood in the sense of a.e. equality? 
The answer is positive as shown in the seminal paper by F.~Rampazzo and H.J.~Sussmann~\cite{RampSussmann07}.  

The problem is therefore what happens when one descends in regularity of $V^i$ further beyond Lipschitz continuity. In fact, there are many cases when Cauchy problems~\eqref{eq_FrobCauchy1} admit unique solutions for every or just for a.e. intial datum.
For instance, as shown in~\cite[corollary 5.2]{caravenna2018directional}, for a vector field $V\in W^{1,p}(\R^d;\R^d)$ with $p>d$ and  $\div V$ bounded, a unique flow $x\mapsto F_t(x)$ is defined for a.e.\ $x\in\R^d$ and all $t\in\R$ as $F_t(x):= y(t)$ where $y(\cdot)$ is the unique solution to the ODE
\[\dot{y}=V(y)\]
satisfying $y(0)=x$. Further, even if $p\leq d$, then a solution to the latter ODE may be not unique for a.e. initial datum,
but in this case there is a natural selection of such solutions called {\em regular Lagrangian flow}~\cite{Ambrosio04_transpBV}. 
Moreover, there are many other cases besides Sobolev regularity when the vector field $V$ admits a regular Lagrangian flow.
In view of the Rampazzo-Sussmann result for Lipschitz vector fields, it is tempting to conjecture that also in the cases
in which the fields $V^i\in W^{1,p_i}(\R^d;\R^d)$ generating unique flows for a.e.\ initial datum (or at least regular Lagrangian flows) 
$F^i$, it is true that
\begin{equation}\label{eq_FrobCommut2}
F_t^1(F_s^2(x))= F_s^2(F_t^1(x))\quad\mbox{for a.e. $x\in \R^d$, whenever } [V^1, V^2] = 0 \quad\mbox{a.e.},
\end{equation}
of course when the Lie bracket  $[V^1, V^2](x)$ is well defined for a.e.\ $x\in \R^d$.
We will show however by means of a counterexample that in general this is false. Nevertheless we are able to show (Theorem~\ref{thm:frobenius} and Theorem~\ref{thm:frobenius_local_Lploc}) that this conjecture is true, if one of the vector fields is Lipschitz (which generalizes the result of~\cite{RampSussmann07}).

The main difficulty here is that when the vector fields $V^i$ are just Sobolev, the respective Lagrangian flows have only very weak regularity properties: in particular, they are just summable but not Sobolev (not even of fractional order)
as shown in~\cite{ACM19,BQ19}, so that their weak derivatives do not have any pointwise (even a.e.) meaning. This is in sharp contrast with the situation when $V^i$ are Lipschitz, since in the latter case it is a textbook result that the respective flows are Lipschitz too. That is why the technique we adopt in this paper is essentially different from that used in~\cite{RampSussmann07}. In fact, in the latter paper a \textit{set-valued Lie bracket} of Lipschitz vector fields is introduced and then it is shown, substantially, that the usual expansions (of course adapted now to the set-valued setting) of compositions of flow maps in each point are still valid for this case with the  set-valued Lie bracket instead of the classical one.
As a side remark, we mention that another quite natural extension of both the classical Lie bracket and its set-valued version,
the measure valued Lie bracket, has been introduced recently in~\cite{CavagnariMarigona18_meascomm}.
In the general case of Sobolev $V^i$, in view of the mentioned lack of regularity of flows, no technique based on pointwise
expansions would work and therefore we are forced to develop a completely different PDE/measure theory-style approach.

\section{Notation and preliminaries}

\subsection{General notation} 
The finite-dimensional space $\R^d$ is assumed to be equipped
with the Euclidean norm $|\cdot|$;
the notation $B_r(x)\subset \R^d$ stands for the usual open Euclidean ball of radius
$r$ centered at $x$. 
In general, the norm in the normed space $E$ will be denoted by $\|\cdot\|_E$.

All the measures over a metric space
considered in the sequel are positive Radon measures, not necessarily finite.
The notation $\mathcal{L}^d$ stands for the Lebesgue measure in $\R^d$. 
If $\mu$ is a measure over a metric space $X$, then
for a Borel
map $T\colon X\to Y$ between metric spaces $X$ and $Y$ we denote
by $T_{\#}\mu$ the push-forward of $\mu$, i.e.\ the measure over
$Y$ defined by
$(T_{\#}\mu)(B):= \mu(T^{-1}(B))$ for every Borel $B\subset Y$.

We use the notation $A \lesssim B$  for possibly vector valued functions $A$, $B$ defined on a subset $\Omega \subseteq \R^d$, when there exists a constant $C \ge 0$ such that
\[ |A|(x) \le C |B|(x) \quad \text{for every $x \in \Omega$}\]
with some constant $C>0$ depending possibly on parameters such as dimension of the space or integrability exponents, but not on $A$ and $B$. 

\subsection{Spaces} For a measure $\mu$ in a metric space $E$
we denote by $L^p(E, \mu ;\R^m)$ 
the usual Lebesgue space of $\mu$-integrable with exponent $p\geq 1$
functions $f\colon E\to \R^m$ ($\mu$-essentially bounded when $p=+\infty$); the reference to $\mu$ will be omitted in the
case $\mu=\mathcal{L}^d$. The reference to $E$ may be also omitted when no confusion is possible.   
Analogously, $W^{k,p}(\R^d;\R^m)$ (resp.\ $W^{k,p}_{\loc}(\R^d;\R^m)$) will stand for the usual Sobolev  (resp.\ locally Sobolev) class of functions over $\R^d$ with values in $\R^m$.
In all the cases the reference to $\R^m$ will be omitted when $m=1$, i.e. for real valued functions. 
The norm in $L^p(\R^d)$ is denoted for brevity just $\|\cdot\|_p$.

The space $C([a,b];\R^d)$ of continuous curves in $\R^d$ parameterized over the interval $[a,b]$ is endowed with the usual supremum norm. For every $t\in [a,b]$ we let $e_t\colon C([a,b];\R^d)\to \R^d$ stand for the evaluation map
$e_t(\theta):=\theta(t)$.

\subsection{Variation} Given a function
\[ \omega\colon [-T,T]^2 \to \R, \quad (s,t) \mapsto \omega_{st},\]
we write
\[ [\omega] := \limsup_{|\pi| \to 0} \sum_{t_i \in \pi} |\omega_{t_{i} t_{i+1}}|.\]
This generalizes the notion of total variation (which corresponds to the case $\omega_{st} := f_t - f_s$).

\subsection{Maximal functions}
We recall the usual definitions of the sharp maximal function $g^\sharp_{r}$ for an $r>0$ and the maximal function 
$g^*$
of a locally integrable function $g\colon \R^d\to \R$, namely, 
\begin{align*}
g^\sharp_{r} (x) & := \sup_{ s \in [0,r]} \frac{1}{\mathcal{L}^d(B_s(x))} \int_{B_s(x)}|g(y) - g(x) | \d y,\\
g^*(x) & := \sup_{ s \ge 0 } \frac{1}{\mathcal{L}^d(B_s(x))} \int_{B_s(x)}|g(y)| \d y,
\end{align*}
so that in particular
\begin{equation}\label{eq_sharpmax_est1}
g^\sharp_r \le 2 g^*.
\end{equation}
\begin{remark}[convergence of maximal functions]\label{rem:sharp-maximal}
	We notice that, if $g \in L^p(\R^d)$ with $p>1$, then 
	\begin{equation}\label{eq:maximal} g^\sharp_r, g^* \in L^p(\R^d), \end{equation}
	and $g^\sharp_r(x)\searrow 0$  as $r \searrow 0$ for a.e.\ $x \in \R^d$ (this can be proven by density of continuous functions). In particular,
	\[ \nor{ g^\sharp_r}_p \searrow 0 \quad\mbox{as $r \searrow 0$}.\]
\end{remark}

\subsection{Pointwise inequalities for Sobolev functions}
Recall that for Sobolev functions $f \in W^{1,1}_{\loc}(\R^d)$, one has the following pointwise inequality~\cite[corollary 1]{bojarski1993pointwise}, valid for every $x$, $y \in \R^d \setminus N$, where $N$ is Lebesgue-negligible: 
\begin{equation}\label{eq:pointwise-sobolev-sharp}
 |f(y)- f(x) - \DD f(x) (y-x) | \lesssim  |y-x| \bra{ |\DD f|^\sharp_{|y-x|}(y)+ |\DD f|^\sharp_{|y-x|}(x)}.
\end{equation}
From~\eqref{eq_sharpmax_est1} we obtain
\begin{equation}\label{eq:pontiwise-sobolev-star}
 |f(y)- f(x) | \lesssim  |y-x| \bra{ |\DD f|^*(y)+ |\DD f|^*(x)}.
\end{equation}
If $f \in W^{2,1}_{\loc}(\R^d)$, then~\cite[corollary 1]{bojarski1993pointwise} gives also
\begin{equation}\label{eq:pointwise-sobolev-second-order}
|f(y)- f(x) - \DD f(x) (y-x) | \lesssim  |y-x|^2 \bra{ |\DD^2 f|^*(y)+ |\DD^2 f|^*(x)}.
\end{equation}

\subsection{Regular Lagrangian flows}

\begin{definition}
We say
	that a Borel map $F \colon [a, b] \times \R^d\to \R^d$ is a regular Lagrangian flow for the 
	(possibly time-dependent) vector field $V\colon [a, b] \times \R^d\to \R^d$
	if
	\begin{itemize}
		\item[(i)] for a.e. $x \in\R^d$ the map $t \in [a,b]\mapsto  F(t, x)$ is an absolutely continuous
		solution of the ODE
\[\dot{y}=V(t, y)\]
for $t\in [a,b]$,
satisfying $y(a)=x$; 		
		\item[(ii)] there is a constant $C>0$ independent of $t$ (called the compressibility constant of $F$) such that
		\[F(t,\cdot)_{\#} \mathcal{L}^d \leq C \mathcal{L}^d\]
		for every $t \in [a, b]$.
	\end{itemize}
\end{definition}

For flows we usually write $F_t(x)$ instead of $F(t,x)$.
We will further use a couple of obsersvations.

\begin{remark}\label{rm_FrobCompr1}
	When $\nu$ is a measure over $\R^d$ with bounded density with respect to the Lebesgue measure, i.e.\
	\[
	\nu= f \mathcal{L}^d\quad\mbox{with $f\in L^\infty(\R^d)$},
	\]
	and $F \colon [a, b] \times \R^d\to \R^d$ is a regular Lagrangian flow with compressibility constant $C$ 
	for some vector field, then
	one has 
	\begin{equation}\label{eq_compress1}
	F(t,\cdot)_{\#} \nu \leq C\|f\|_\infty  \mathcal{L}^d,
	\end{equation}
	because $F(t,\cdot)_{\#} \nu= (f\circ F(t,\cdot)^{-1})  F(t,\cdot)_{\#} \mathcal{L}^d$.
\end{remark}

\begin{lemma}\label{lm_LagrFlowNeglset}
If $F \colon [a, b] \times \R^d\to \R^d$ 
is a regular Lagrangian flow and $N\subset \R^d$ is Lebesgue negligible, then
there is a $B_0\subset \R^d$ with $\mathcal{L}^d(B_0)=0$ such that
$F(t,x)\not\in N$ for every $x\not\in B_0$ and 
a.e. $t\in [a,b]$.  
\end{lemma}

\begin{proof}
For the set $G:=\{(x,t)\in [a,b]\times\R^d\colon F(t,x)\in N\}$ 
one has $\mathcal{L}^{d+1}(G)=0$ by Fubini theorem since for each fixed $t\in [a,b]$ one has
\begin{align*}
\mathcal{L}^d(G\cap(\{t\}\times\R^d))=(F(t,\cdot)_{\#} \mathcal{L}^d)(N) \leq C \mathcal{L}^d (N)=0.
\end{align*}
Therefore the set
\[
B_0:=\{x\in \R^d\colon \mathcal{L}^1(G\cap([a,b]\times\{x\}))>0\}\subset \R^d
\]
is Lebesgue negligible as claimed.
\end{proof}


\section{Vanishing Lie bracket does not imply commutativity of flows}\label{sec:counterexample}

In this section we show by means of the example below that when two vector fields have a.e. vanishing Lie bracket, it is in general not true that their flows commute a.e.\ even if we assume that they are a.e. uniquely defined, and even if they are regular Lagrangian. The main idea comes from an example \cite[Nelson's example]{Reed-Simon} of non-commutativity of groups with generators commuting on a common dense set. In our case, we consider a foliation of $\R^d$ into helix-like hypersurfaces so that the composition of the two flows yields a rotation along the axis of the helix. 

\begin{example}
Let $d :=3$,
\[ f(x,y) := \arctan\bra{\frac{y}{x}}\]
and consider the vector fields
\begin{align*}
V^1 &:=\partial_{x} + (\partial_{x}f) \partial_z = \partial_{x} - \frac{y}{x^2 + y^2} \partial_z,\\ 
 V^2 & := \partial_{y} + (\partial_{y}f) \partial_z= \partial_{y} + \frac{x}{x^2 + y^2}\partial_z,
\end{align*}
so that
\begin{equation}\label{eq_Lie0a}
[V^1, V^2] =\bra{ \partial_{xy}^2 f -  \partial_{yx}^2 f } \partial_z = 0,
\end{equation}
i.e.\ their Lie bracket vanishes
(of course the function $f$, the vector fields and their Lie bracket are defined everywhere in $\R^d$ outside of the two-dimensional plane $\{x=0\}$, hence a.e.\ in $\R^d$). 
We now verify the sequence of claims.

{\sc Claim 1}. Both $V^1$ and $V^2$ are tangent to the level sets of the function 
$(x,y,z)\in \R^d \mapsto z-f(x,y)\in \R$, i.e.\ to the
graphs of the functions $z=f(x,y)+C$ in $\R^d$, with $C$ any constant, as observed by a direct calculation.

{\sc Claim 2}. The flows $F^i$ for the vector fields $V^i$, $i=1,2$, are uniquely defined over $[-T,T]$, for every $T>0$ and for every initial datum outside of the plane $\{x=0\}$.
In fact, this is quite immediate for $F^2$, since the solution to the ODE $\dot\theta= V^2(\theta)$ is uniquely defined for every
initial datum $\theta(0)\in \R^d\setminus \{x=0\}$  over all times.

As for $F^1$, we consider the solutions to the ODE $\dot\theta= V^1(\theta)$ for every initial datum outside of the plane $\{x=0\}$.
If $\bar x:= \theta_1(0)<0$ then the solution $\theta(\cdot)$ to the respective Cauchy problem is defined uniquely for all $t\in (-\infty, -\bar x)$,
and, analogously, if $\bar x:= \theta_1(0)>0$ then it is defined uniquely for all $t\in (-\bar x, +\infty)$.
In the first case,  one has
\begin{equation*}
\lim_{t\to -\bar x -0} f(\theta_1(t), \theta_2(t))= 
\left\{
\begin{array}{rl}
\pi/2, &\mbox{if $\bar y<0$},\\
0, & \mbox{if $\bar y=0$},\\
-\pi/2, &\mbox{if $\bar y>0$},
\end{array}
\right.
\end{equation*} 
because $\theta_2(t)=\theta_2(0)$ for all $t\in (-\infty, -\bar x)$ and 
\[
\lim_{t\to -\bar x -0}\theta_1(t)=0^-.
\] 
Thus, denoting $\bar z:= \theta_3(0)$ and recalling Claim~1, we get that
\begin{align*}
\lim_{t\to -\bar x -0}\theta_3(t) &=\lim_{t\to -\bar x -0} f(\theta_1(t), \theta_2(t))
+\left(\bar z - f(\bar x, \bar y)\right)\\
&=
\bar z - f(\bar x, \bar y) -\frac{\pi}{2} \mathrm{sign}\,\bar y .
\end{align*}
Symmetrically in the second case, $\theta_2(t)=\theta_2(0)$ for all $t\in (-\bar x, +\infty)$,
\[
\lim_{t\to -\bar x +0}\theta_1(t)=0^+
\] 
and
\begin{align*}
\lim_{t\to -\bar x +0}\theta_3(t) &=
\bar z - f(\bar x, \bar y) -\frac{\pi}{2} \mathrm{sign}\,\bar y .
\end{align*}
Summing up, we have that absolutely continuous solutions $\theta(\cdot)$ to the respective Cauchy problem are
defined uniquely for every initial datum 
outside of the plane $\{x=0\}$ and for a.e.\ $t\in \R$, or, more precisely, for every $t\in \R$ except $t=-\bar x$,
with
\begin{equation}\label{eq_limThet1}
\lim_{t\to -\bar x}\theta(t) =
\left( 0,\bar y, \bar z - f(\bar x, \bar y) -\frac{\pi}{2} \mathrm{sign}\,\bar y \right).
\end{equation}
This defines uniquely the flow $F^1$.

{\sc Claim 3}. 
One has $\div V^i=0$, and hence the flows $F^i$
are regular Lagrangian (because $F^i_{t\#}\mathcal{L}^d=\mathcal{L}^d$ for all $t$ for which $F^i_{t}$ is defined, i.e.\ the compressibility constant is one),  for $i=1,2$. 

{\sc Claim 4}. Finally, the flows $F^1$ and $F^2$ do not commute despite~\eqref{eq_Lie0a}.
In fact, consider an arbitrary
$p=(\bar x, \bar y,\bar z)\in \R^d$ with $x<0$ and $y<0$. Then for $\theta^1(t):=F^1_t(p)$ one has
 \begin{equation}\label{eq_limThet2}
\theta^1(t) =
\left( \bar x +t, \bar y,\theta^1_3(t)\right)
\end{equation}
with $\theta^1_3(-\bar x)=\bar z - f(\bar x, \bar y) +\pi/2$ in view of~\eqref{eq_limThet1}. By Claim~1 one has that $\theta^1(t)$ 
for $t>-\bar x$
belongs to the graph of function $z=f(x,y) + C$, which means 
\[
\theta^1_3(t)=f(\bar x+t,\bar y) + C
\]
for all $t>-\bar x$. Letting $t\to -\bar x+0$, we get
\[
\bar z - f(\bar x, \bar y) +\pi/2=\theta^1_3(-\bar x)= -\pi/2+C,
\]
which gives $C=\bar z - f(\bar x, \bar y) +\pi$. Therefore, 
$F^2_s(F^1_t(p))$ belongs to the graph of function
\[
z=f(x,y) + \bar z - f(\bar x, \bar y) +\pi
\]
for every $s\in \R$ and $t>-\bar x$.

On the other hand, for $\theta^2(s):=F^2_s(p)$ one has 
\[
\theta^2_2(s)= \left(\bar x, \bar y+s, , \theta^2_3(s) \right),
\] so that
for every $s> -\bar y$ one has $\theta^2_2(s)>0$. 
Then for $\sigma^1(t):=F^1_t(\theta^2(s))$ 
for $s>-\bar y$ 
one has
\begin{equation}\label{eq_limThet3}
\sigma^1(t) =
\left( \bar x+t, \bar y+s,  \sigma^1_3(t) \right).
\end{equation}
Since $\theta^2(s)$ belongs to the graph of the function
\[
z=f(x,y) + \bar z - f(\bar x, \bar y) 
\]
by Claim~1, then so is $\sigma^1(t)$ for $t<-\bar x$. 
By~\eqref{eq_limThet1} one has
\begin{equation}\label{eq_limThet4}
\sigma^1(-\bar x) =
\left( \bar x, \bar y+s,  \bar z - f(\bar x, \bar y) -\frac{\pi}{2} \right).
\end{equation}
Hence for $t>-\bar x$ (and $s>-\bar y)$ one has that
$F^1_t(F^2_s(p))=\sigma^1(t)$ belongs to the graph of function $z=f(x,y) + C$, that is, 
\[
\sigma^1_3(t)=f(\bar x+t,\bar y) + C
\]
for all $t>-\bar x$, and the value of $C$ can be found by letting $t\to -\bar x+0$, since then by~\eqref{eq_limThet4}
we have
\[
\bar z - f(\bar x, \bar y) -\pi/2=\sigma^1_3(-\bar x)= \pi/2+C,
\]
which gives $C=\bar z - f(\bar x, \bar y) -\pi$. Therefore, 
$F^1_t(F^2_s(p))$ belongs to the graph of function
\[
z=f(x,y) + \bar z - f(\bar x, \bar y) -\pi
\]
for every $s> -\bar y$ and $t>-\bar x$.
In other words for all such pairs $(s,t)$ one has
\[
F^1_t(F^2_s(p))\neq F^2_s(F^1_t(p))
\]
as claimed.
\end{example}


\begin{remark}
One can easily check that both $V^1$, $V^2$ are not  Sobolev (even locally) on $\R^3$, since for $f(x,y):= \arctan(y/x)$ we have $f \notin W^{2,1}_{\loc}(\R^2)$. We may in fact prove more, that is, any two two vector fields on $\R^3$ of the form
 \[ V^1 =\partial_{x} + \partial_{x}f \partial_z, \quad V^2 = \partial_{y} + \partial_{y}f \partial_z,\]
 with $f= f(x,y) \in W^{2,1}_{loc}(\R^2)$ must have commuting flows $F^1$, $F^2$. Indeed, they are explicitly given by
\begin{align*}
F^1_s( x,y,z) & :=\left( x+s, y, z + \int_0^s \partial_{x}f(x+r, y) d r\right),\\
F^2_t( x,y,z) & :=\left( x, y+t, z + \int_0^t \partial_{y}f(x, y+r) d r\right).
\end{align*}
If we introduce the differential $1$-form on $\R^{3}$
\[ \omega := df = \partial_{x} f  d x + \partial_{y} f  d y,\]
then $F^1_s \circ F^2_t  = F^2_t \circ F^1_s$, if and only if 
\[ \int_{\partial R} df  = 0\]
for ``almost every'' oriented rectangle $R \subset \R^2$ with parallel sides to the coordinate axes and side lengths $s$, $t$ (namely, for every rectangle of the form $[x, x+s] \times [y, y+t]$ for a.e. $(x,y)\in \R^2$). 
Note that $\omega$ is independent on $z$, and, moreover,
if $f \in W^{2,1}_{\loc}(\R^2)$, then the components of $\omega$ belong to $W^{1,1}_{\loc}(\R^2)$, so that we can apply Stokes theorem to deduce
\[  \int_{\partial R} \omega=\int_{\partial R} df = \int_R d^2 f = 0,\]
thus the two flows commute, i.e.\ to be more precise
\[F^1_t(F^2_s(x,y,z))=F^2_s(F^1_t(x,y,z))\]
 for a.e. $(x,y)\in \R^2$, all $z\in \R$ and $(s,t)\in \R^2$. 
 \end{remark}

\section{Vanishing Lie bracket and commutativity of regular Lagrangian flows for Sobolev vector fields}

We prove now that regular Lagrangian flows $F^1$, $F^2$ of Sobolev vector fields $V^1$, $V^2$ respectively commute a.e., if the Lie bracket of the latter vanishes a.e., i.e.~\eqref{eq_FrobCommut2} holds, once  one of the two  is Lipschitz continuous.

\begin{theorem}\label{thm:frobenius}
Let $p\in [1, \infty]$, 
and assume that
$ V^1 \in W^{1,p}(\R^d)$ 
is such that its regular Lagrangian flow $F^1_s$ is uniquely defined for $s$ in some interval containing zero, and
$V^2 \in W^{1,\infty} (\R^d)$.  
Let $(F^1_t)_{t \in \R}$, $(F^2_t)_{t \in \R}$ be the corresponding regular Lagrangian flows. If
\[ [V^1, V^2](x) := (\DD V^2)(x) V^1(x) -  (\DD V^1)(x) V^2(x) = 0\quad\mbox{for a.e.\ $x \in \R^d$},\]
then the flows commute, i.e., for a.e.\ $x \in \R^d$, it holds
\[ F^2_t \bra{ F^1_s (x) }= F^1_s\bra{ F^2_t (x)}, \quad \text{for every $s$, $t \in \R$.}\]
\end{theorem}

\begin{proof}
Since $V^2$ is Lipschitz continuous, its regular Lagrangian flow coincides a.e.\ with the classical flow. Therefore, we may assume that $x \mapsto F^2_t(x)$ is defined everywhere and Lipschitz continuous. It is also  sufficient to prove the thesis for $s$, $t \in [-T, T]$ for an arbitrary $T>0$ such that both $F^i_t$ are defined for $t\in [-T,T]$, $i=1,2$. 

Let $\nu$ stand for a finite measure on $\R^d$ with bounded and strictly positive density with respect to $\mathcal{L}^d$
(the $L^\infty$ norm of the latter will be denoted by $\|\nu\|_\infty$), e.g.\ a standard Gaussian.
For  $x \in \R^d$ and $s, t \in [-T,T]$ we set
\begin{equation}\label{eq_Xst1a}
X_{s,t}(x) := F^2_t( F^1_s(x)).
\end{equation}
Since the flow $F^2$ is Lipschitz continuous, we have that for a.e.\ $x \in \R^d$ and for every $s, t \in [-T,T]$, 
the map $(s,t) \mapsto X_{s,t}(x)$ is continuous. Moreover, for a.e.\ $x \in \R^d$ one has 
\begin{align*}
\partial_t X_{s,t}(x)  = V^2( X_{s,t}(x) )  & \quad \text{for a.e.\ $t \in (-T,T)$,}\\
\partial_s X_{s,0}(x) = V^1(X_{s,0}(x))  & \quad \text{for a.e.\ $s \in (-T,T) $.}
\end{align*} 

Now, for $s$, $s'$, $t \in [-T,T]$, we write
\[ A_{s, s'; t}(x)  :=  X_{s',t}(x) - X_{s,t}(x), \quad  B_{s, s'; t}(x)  :=  A_{s, s'; t}(x) -  (s'-s) V^1( X_{s,t}(x) ),\]
frequently omitting the reference to $x$ for brevity, and for every $t\in [-T,T]$ we consider the 
measure 
$\eta:=\hat{X}_{t\#}\nu$ over $C([-T,T],\R^d)$, where
$\hat{X}_t\colon \R^d\to C([-T,T],\R^d)$ is defined by the formula
\[
\hat{X}_t(x):= X_{(\cdot, t)}(x).
\]
Note that $\eta$ is concentrated over curves $ [-T,T] \ni s \mapsto X_{s,t}(x)$ for a set of full $\nu$-measure of $x\in \R^d$.  
Further,
\[
e_{s\#}\eta = X_{s,t\#}\nu =  F^2_{t\#} (F^1_{s\#} \nu) \leq C_2C_1\mathcal{L}^d,
\]
where $C_i$ are the compressibility constants for $F^i$, $i=1,2$.

We will show that for some function $(\omega_{ss'})_{s, s' \in [-T,T]}$, with $[\omega] = 0$ one has
\begin{equation}\label{eq_Frob_main1a}
\nor{B_{s,s';t}(\cdot)}_{L^1(\nu)} \lesssim \omega_{ss'}, \quad \text{for $s, s' \in [-T,T]$.}
\end{equation}
This would give
\[
\nor{\theta_{s'}-\theta_s - (s'-s) V^1(\theta_s)  }_{L^1(C([-T,T],\R^d);\eta)}  \lesssim \omega_{ss'}
\]
for every $s, s' \in [-T,T]$, which means in view of Proposition~\ref{prop:taylor-first-order} 
(applied with $V^1$ instead of $V$ and $C_1C_2$ instead of $C$) that 
$\eta$-a.e.\ $\theta\in C([-T,T],\R^d)$ is an integral curve of $V^1$. In other words, for
$\nu$-a.e.\ (hence for $\mathcal{L}^d$-a.e., because  the density of $\nu$ is strictly positive) 
$x\in \R^d$ the curve $\theta$ defined by $\theta(s):=X_{s,t}(x)$ satisfies $\partial_s\theta= V^1(\theta(s))$ for every
$s\in [-T,T]$. Since for such curves one has $\theta(0)=F^2_t(x)$, then for a.e. $x\in \R^d$ one has
\[X_{s,t}(x)=F^1_s(F^2_t(x)),\] which is the claim of the theorem being proven.

The rest of the proof is therefore dedicated to showing the claim~\eqref{eq_Frob_main1a}. It will be done in six steps.

\emph{Step 1: Estimate on $A_{s,s';0}$.} 
In the subsequent estimates we suppose without loss of generality that $s'\geq s$ (the case $s\geq s'$ is completely symmetric).
One has 
\begin{equation*}\label{eq_estAs0a}
\begin{aligned}
|A_{s,s';0}(x) |^{p}&=|X_{s',0}(x) - X_{s,0}(x)|^{p}= |F^1_s(x)-F^1_{s'}(x)|^{p} \\
 &\leq \left| \int_s^{s'} |V^1(y(\tau))|\,d\tau \right|^{p} \\
 & \leq |s-s'|^{p/p'}\int_s^{s'} |V^1(y(\tau))|^{p}\,d\tau\quad\mbox{by H\"{o}lder inequality}.
\end{aligned}
\end{equation*}
Integrating with respect to $\tilde{\eta}:=\nu\res B\otimes \delta_{F^1(\cdot,x)}$, where $B\subset\R^d$ is an arbitrary fixed Borel set, gives
\begin{equation}\label{eq_estAs0b}
\begin{aligned}
\int_{B} |A_{s,s';0}(x) |^{p}\d\nu(x)& \leq |s-s'|^{p/p'}\left| \int_{C([-T,T];\R^d)} \, \d\tilde{\eta}(y)\int_s^{s'} |V^1(y(\tau))|^{p}\,\d\tau\right| \\ 
&= |s-s'|^{p/p'}\left|\int_s^{s'} \d\tau \int_{C([-T,T];\R^d)} \, \d\tilde{\eta}(y) |V^1(y(\tau))|^{p}\right|\\
& = |s-s'|^{p/p'}\left|\int_s^{s'} \,\d\tau \int_{\R^d} \, \d e_{\tau\#}\tilde{\eta}(x) |V^1(x)|^{p}\right|\\
&\leq C\|\nu\|_\infty |s-s'|^{1+p/p'} \int_{F^1_\tau(B)} |V^1(x)|^p\, \d x,
\end{aligned}
\end{equation}
because $e_{\tau\#}\tilde{\eta}= F^1_{\tau\#}(\nu\res B)= \mathbf{1}_{F^1_\tau(B)}F^1_{\tau\#}\nu
$ and in view of~\eqref{eq_compress1}. In particular, with $B:=\R^d$ we get
\begin{equation}\label{eq_estAs0c}
\begin{aligned}
\|A_{s,s';0}(x) \|_{L^{p}(\nu)}& \lesssim |s-s'|^{1/p+1/p'}\left\|V^1\right\|_{p} = |s-s'|\left\|V^1\right\|_{p}. 
\end{aligned}
\end{equation}
Further, since 
$F^1_\tau(B)=(F^1_{-\tau})^{-1}(B)$,
we get 
\[
\mathcal{L}^d(F^1_\tau(B))=(F^1_{-\tau\#}\mathcal{L}^d)(B)\leq C\mathcal{L}^d(B),
\]
and hence from~\eqref{eq_estAs0b} with $p=1$ we get that the functions $A_{s,s';0}(\cdot)/|s-s'|$ are equiintegrable
in $L^1(\nu)$, if $V^1\in L^1(\R^d)$.


\emph{Step 2: Estimate on $A_{s,s';t}$.} 
Lipschitz continuity of $F^2_t$ gives, for a.e.\ $x \in \R^d$,
\[\partial_t |A_{s, s'; t}| \le | V^2( X_{s,t}) - V^2( X_{s',t})| \le  \nor{\DD V^2}_\infty | A_{s,s'; t} |, \quad \text{for a.e.\ $t \in [-T,T]$.}\]
Hence, by Gronwall lemma,
\[ \sup_{t \in [-T,T]} |A_{s,s'; t}|  \le |A_{s,s'; 0}|\exp\bra{ \nor{\DD V^2}_\infty T}. 
\]
Integrating with respect to $\nu$ gives
\begin{equation}\label{eq:gronwall-a} \begin{split} 
\sup_{t \in [-T,T]} \nor{   |A_{s,s'; t}| }_{L^{p}(\nu)} &\le
\nor{  \sup_{t \in [-T,T]} |A_{s,s'; t}| }_{L^{p}(\nu)} \\
&\le   \nor{  A_{s,s';0} }_{L^{p}(\nu)}\exp\bra{ \nor{\DD V^2}_\infty T}\\
& \lesssim |s'-s| \nor{V^1}_{p}  \exp\bra{ \nor{\DD V^2}_\infty T}
\quad\mbox{in view of~\eqref{eq_estAs0c}}\\
& \lesssim |s'-s|.
\end{split}\end{equation}

\emph{Step 3: Integral inequality for $B_{s,s';t}$}. 
We first prove that 
\begin{itemize}
	\item[(i)] for a.e.\ $x \in \R^d$ and for every $s \in [-T,T]$ the curve
	$ t \mapsto V^1(X_{s,t}(x))$ 
	is absolutely continuous, and 
	\item[(ii)] satisfies 
	\begin{equation}\label{eq_FrobDV1a}
	\partial_t V^1(X_{s,t}(x)) = \DD V^1( X_{s,t}(x)) V^2(X_{s,t}(x)).
	\end{equation}
\end{itemize}
In fact, for a.e.\ $x \in \R^d$ and for every $s \in [-T,T]$ the curve
$t \mapsto X_{s,t}(x)$ is absolutely continuous (even Lipschitz). 
Letting 
\begin{align*}
\nu^1_s := F^1_{s\#}\nu, \quad \eta^2_s:= F^2_{\#} \nu^1_s,
\end{align*}
we have that $\eta^2_s$ is a measure on $C([-T,T];\R^d)$ concentrated over such curves, while
\[
\nu^1_s\leq C_1 \|\nu\|_\infty \mathcal{L}^d
\]
by Remark~\ref{rm_FrobCompr1}, and therefore
\[
e_{t\#}\eta^2_s=F^2_{t\#} \nu^1_s\leq C_2C_1 \|\nu\|_\infty\mathcal{L}^d
\]
by the same Remark. Therefore,~(i) and ~(ii) follow from Lemma~\ref{lm_chainrule_sobolev_curves} with $\eta^2_s$ in place of $\eta$ and
$V^1$ in place of $V$.


Note that since $[V^1,V^2](x)=0$ for a.e. $x\in \R^d$, then by Lemma~\ref{lm_LagrFlowNeglset} there 
is a $B_0\subset \R^d$ with $\mathcal{L}^d(B_0)=0$ such that
$[V^1,V^2](X_{s,t}(x))=0$ for every $x\not\in B_0$ and 
a.e. $s,t\in [-T,T]$.  

Using~\eqref{eq_FrobDV1a} we obtain
\[ \begin{split} \partial_t  B_{s, s'; t} & = \partial_t A_{s, s'; t} - (s'-s) \DD V^1( X_{s,t} ) V^2 (X_{s,t}), \\
& = \partial_t A_{s,s'; t} - (s'-s) \DD V^2( X_{s,t} ) V^1 (X_{s,t}) \quad \text{ (since $[V^1, V^2](X_{s,t}) =0$ a.e.)}\\
& = \DD V^2( X_{s,t}) B_{s, s'; t} + R_{s,s'; t},
\end{split}\]
where 
\[ R_{s,s'; t} := V^2( X_{s',t}) - V^2( X_{s,t}) - \DD V^2(X_{s,t}) ( X_{s',t} - X_{s,t}).\] 
At this point Gronwall lemma yields
\begin{equation}\label{eq_frobEstB1a}
\sup_{t \in [-T,T]} |B_{s, s'; t}(x)|  \le \bra{ |B_{s, s'; 0}(x)| + \int_{-T}^T  |R_{s,s';\tau}(x)| \d \tau} \exp\bra{ \nor{\DD V^2}_\infty T},
\end{equation}
which implies
\begin{equation}\label{eq_frobEstB1b}
\sup_{t \in [-T,T]} \|B_{s, s'; t}\|_{L^1(\nu)}  \le \bra{ \|B_{s, s'; 0}\|_{L^1(\nu)} + 
	\int_{-T}^T  \|R_{s,s';\tau}\|_{L^1(\nu)} \d \tau} \exp\bra{ \nor{\DD V^2}_\infty T}.
\end{equation}

\emph{Step 4: Estimate on $B_{s,s';0}$.} To estimate $\|B_{s,s'; 0}\|_{L^1(\nu)}$ we use Proposition~\ref{prop:taylor-first-order} with $V^1$ instead of $V$, $p_0:=p$, $p_1:=p'$, $q:=1$, obtaining
\begin{equation}\label{eq_frobEstB1c}
\nor{ B_{s, s'; 0}}_{L^1(\nu)} \lesssim |s'-s|^2.
\end{equation}

\emph{Step 5: Estimate on $R_{s,s';\tau}$.} We use~\eqref{eq:pointwise-sobolev-sharp} with the components of $V^2$ instead of $f$,
getting
\[ \abs{ V^2(y) - V^2(z) - \DD V^2(z) ( y-z)} \lesssim |y-z| \bra{ (\DD V^2)^\sharp_{|y-z|}(y) +(\DD V^2)^\sharp_{y-z}(z)},\]
and choose $z := X_{s,t}(x)$, $y := X_{s',t}(x)$, so that
\[
R_{s,s';t}(x)  \lesssim A_{s,s';t}(x) \bra{ (\DD V^2)^\sharp_{|A_{s,s';t}(x)|}(X_{s',t}(x)) +(\DD V^2)^\sharp_{|A_{s,s';t}(x)|}(X_{s,t}(x))}.
\]
Integrating the above inequality over $x\in \R^d$ with respect to $\nu$ and dividing by $|s-s'|$, we arrive at the estimate 
\begin{equation}\label{eq_RAD1a}
\begin{split} \int_{-T}^T &\frac{\nor{ R_{s,s';\tau}}_{L^1(\nu)}}{|s-s'|}\d \tau   \lesssim \int_{-T}^T \d \tau\int_{\R^d}\frac{A_{s,s';\tau}(x)}{|s-s'|}D_{s,s';\tau}(x) \d \nu (x), \quad\mbox{where}\\
& D_{s,s';\tau}(x):=\bra{ (\DD V^2)^\sharp_{|A_{s,s';\tau}(x)|}(X_{s',\tau}(x)) +(\DD V^2)^\sharp_{|A_{s,s';\tau}(x)|}(X_{s,\tau}(x))}.
\end{split}
\end{equation}
We will show that 
one has
\begin{eqnarray}
\label{eq_FrobDV2int1a}
\int_{-T}^T \d \tau \int_{\R^d} 
(\DD V^2)^\sharp_{|A_{s,s';\tau}(x)|}(X_{s',\tau}(x)) 
\d\nu (x) \to 0,\\
\label{eq_FrobDV2int1b}
\int_{-T}^T \d \tau \int_{\R^d} 
(\DD V^2)^\sharp_{|A_{s,s';\tau}(x)|}(X_{s,\tau}(x)) 
\d\nu (x)  \to 0,
\end{eqnarray}
which means in particular that the functions $\tilde D_{s,s'}$ defined by $(\tau, x)\mapsto D_{s,s';\tau}(x)$ converge to zero 
in $L^1(\R^d\times(-T,T),\nu\otimes \mathcal{L}^1)$, hence
in measure
$\nu\otimes \mathcal{L}^1\res[-T,T]$
as $\delta:=s'-s\to 0^+$, uniformly in $s\in [-T,T]$. 
Minding that these functions are uniformly bounded (since the maximal function
$(\DD V^2)^\sharp$
is bounded by $\|DV^2\|_\infty$), and the functions $\tilde A_{s,s'}$ defined by $(\tau, x)\mapsto A_{s,s';\tau}(x)/|s-s'|$
are uniformly bounded in $L^1(\R^d\times(-T,T),\nu\otimes \mathcal{L}^1)$ and equiintegrable by Step~1,
then choosing arbitrary sequences $s_k, s_k'\in [-T,T]$ with $\delta_k:=s_k'-s_k\to 0^+$
from Lemma~\ref{lm_weak_meas_conv1} (with $f_k:=\tilde A_{s_k,s_k'}$, $g_k:=\tilde D_{s_k,s_k'}$, $E:=\R^d\times(-T,T)$, $\mu:=\nu\otimes \mathcal{L}^1$) and
from~\eqref{eq_RAD1a} one gets
\begin{equation}\label{eq_frobEstB1d}
\int_{-T}^T \nor{ R_{s,s'; \tau}}_{L^1(\nu)}\d \tau \lesssim |s'-s| o(1).
\end{equation}
Plugging~\eqref{eq_frobEstB1c} and~\eqref{eq_frobEstB1d} into~~\eqref{eq_frobEstB1b}, we get
\begin{equation*}\label{eq_frobEstB1e}
\sup_{t \in [-T,T]} \|B_{s, s'; t}\|_{L^1(\nu)}  \lesssim  o(|s'-s|)
\end{equation*}
as $s'-s\to 0^+$,
which gives the thesis.

{\em Step 6}. It remains to prove~\eqref{eq_FrobDV2int1a} and~\eqref{eq_FrobDV2int1b}.
In order to verify~\eqref{eq_FrobDV2int1a},
let $\nu_s$ be as in Step~3, i.e.\ $\nu_s := (F^1_s)_\sharp \nu$, so that a change of variables  gives
\[ \int_{\R^d} (\DD V^2)^\sharp_{|A_{s,s';\tau}(\cdot)|}(X_{s',\tau}(\cdot)) \d\nu = \int_{\R^d} (\DD V^2)^\sharp_{|A_{0,s'-s;\tau}(\cdot)|}(X_{s'-s,\tau}(\cdot)) \d\nu_s.
\]
Thus we simply observe that
\[
\int_{\R^d} 
(\DD V^2)^\sharp_{|A_{0,\delta;\tau}(x)|})(X_{\delta,\tau}(x)) 
\d\nu_s(x) \to 0
\]
as $\delta:=s'-s\to 0^+$ uniformly in $s\in [-T,T]$ by Lemma~\ref{lm_domconvmeas1} with
$g_\delta(x):= 
(\DD V^2)^\sharp_{|A_{0,\delta;\tau}(x)|}(X_{\delta,\tau}(x)) 
$
(note that these functions are bounded by $\|\DD V^2\|_\infty$ since so are the maximal functions
$(\DD V^2)^\sharp$ and $\lim_{\delta\to 0^+}g_\delta =0$ by Remark~\ref{rem:sharp-maximal}).
Therefore by Lebesgue dominated convergence theorem we get
\[ \int_{0}^T \d\tau \int_{\R^d}(\DD V^2)^\sharp_{|A_{0,s'-s;\tau}(x)|}(X_{s'-s,\tau}(x))  \d \nu_s(x) \to 0\]
uniformly in $s\in [-T,T]$ as $s'-s\to 0^+$. 
The proof of~\eqref{eq_FrobDV2int1b} is completely analogous, and we can conclude the proof of the theorem.
\end{proof}

We give also the local version of the above Theorem~\ref{thm:frobenius}.

\begin{theorem}\label{thm:frobenius_local_Lploc}
	For $p \in [1, \infty]$ assume that 
	$V^1 \in W^{1,p}_{\loc}(\Omega;\R^d)$ 
	for some open set $\Omega\subset \R^d$ 
	is such that its regular Lagrangian flow $F^1_s(x)$ is uniquely defined over $(-\bar T(x), \bar T(x))$ for some measurable function
	$\bar T\colon \Omega\to (0,+\infty)$,
	and 
	$V^2 \in W^{1,\infty}_{\loc} (\Omega;\R^d)$.
	If 
	\[ [V^1, V^2](x) = 0\quad\mbox{for a.e.\ $x \in \Omega$},\]
	then for a.e. $x \in \Omega$ there is a $T=T(x)>0$ such
	that the regular Lagrangian flows $F^1$, $F^2$ satisfy
\begin{equation}\label{eq_flow_comm1}
	F^2_t \bra{ F^1_s (x) }= F^1_s\bra{ F^2_t (x)} \quad \text{for every $s$, $t \in [-T,T]$.}
\end{equation}
\end{theorem}

\begin{proof}
	For every ball  $B:=B_R(x_0)$ such that $\bar B\subset\Omega$ and every $T_1 >0$ we find a $\rho>R$ with the property that
	$B_\rho(x_0)\subset\Omega$.
	Let $M_{B,T_1}\subset B$
	stand for the (possibly empty) Borel set of points $x\in B$ such that $F^1_s(x)$ is defined for all
	$s\in [-T_1,T_1]$ and stays in $B$. 
By Lemma~\ref{lm_apriori_ODE1} applied with $V^2$ in place of $V$ there is a $T_2>0$ (depending on $B$) such that for all $x\in B$ the classical flow induced by $V^2$ (which is equal to the regular Lagrangian flow $F^2$) is defined for all $t\in [-T_2,T_2]$ and stays inside $B_\rho(x_0)$. We set then $T:=T_1\wedge T_2$, so that
clearly $T=T(B,T_1)$.
In this way the map~\eqref{eq_Xst1a} is defined over $M_{B,T_1}$ for all $t,s\in [-T,T]$, and takes values in $B_\rho(x_0)$.
Note now that in the proof of Theorem~\ref{thm:frobenius} one only evaluates the Lie bracket $[V_1,V_2]$ 
along the trajectories of ODEs with right-hand sides $V_1$, $V_2$. The latter, if started 
at a point of $M_{B,T_1}$ remain in $B_\rho(x_0)$, so that  $[V_1,V_2]$ vanishes along them. 
Thus reiterating the proof of Theorem~\ref{thm:frobenius}
with $\nu:=\mathcal{L}^d(M_{B,T_1})$ we get that
for a.e. $x\in M_{B,T_1}$
the regular Lagrangian flows $F^1$, $F^2$ commute, i.e.~\eqref{eq_flow_comm1} holds. 
Representing now $\Omega$ as a disjoint union of balls $\Omega:=\sqcup_j B_j$ such that $\bar B_j\subset\Omega$ for all $j\in \N$, 
one has that~\eqref{eq_flow_comm1} holds for a.e. $x\in   M_{B_j,1/k}$ with $T=T(B_j, 1/k)$, and since for all $j$
the sets $\{M_{B_j,1/k}\}_{k\in\N}$ cover almost all $B_j$, the claim follows.
\end{proof}

\begin{remark}
	If under conditions of Theorem~\ref{thm:frobenius_local_Lploc} one has additionally that $V^1 \in L^\infty_{\loc}(\Omega)$, 
	then one can assume that $T\in L^\infty_{loc}(\Omega)$. In fact, 
	for every
$x_0\in \Omega$, $R>0$, $\rho>R$ as in the above proof, letting $\rho_1\in (R, \rho)$
apply Lemma~\ref{lm_apriori_ODE1} with $V^1$ instead of $V$ and $\rho_1$ instead of $\rho$ to get a number
$T_1>0$ such that $F^1_s(x)$ is defined for $s\in [-T_1, T_1]$ and $x\in B_R(x_0)$ and maps almost all of 
$B:=B_R(x_0)$ into $B_{\rho_1}(x_0)$ for every $s\in [0, T_1]$. Applying again 
Lemma~\ref{lm_apriori_ODE1} now with $V^2$ in place of $V$ and $\rho_1$ in place of $R$ we get the existence of a $T_2>0$ be such that for all $x\in B_{\rho_1}(x_0)$ the classical flow induced by $V^2$ (which is equal to the regular Lagrangian flow $F^2$) is defined for all $t\in [-T_2,T_2]$ and stays inside $B_\rho(x_0)$. We set then $T:=T_1\wedge T_2$, so that
clearly $T=T(B)$. One has then that
	the	regular Lagrangian flows $(F^1_t)_{t \in [-T,T]}$, $(F^2_t)_{t \in [-T,T]}$
		defined for  a.e. $x\in B$ commute,  i.e.,~\eqref{eq_flow_comm1} holds for every $s$, $t \in [-T,T]$.
The claim follows by representing $\Omega$ as a disjoint union of balls $\Omega:=\sqcup_j B_j$ such that $\bar B_j\subset\Omega$ for all $j\in \N$: in fact, every relatively compact subset of $\Omega$ is covered by only a finite number of $B_j$.
\end{remark}

The following lemma has been used in the above proof.

\begin{lemma}\label{lm_apriori_ODE1}
Let $\Omega\subset \R^d$ be an open set and
the vector field $V\in L^\infty_{\loc}(\Omega;\R^d)$ be such that
the respective regular Lagrangian flow $(t,x)\mapsto F_t(x)$ is defined for 
all $t\in (-\bar T, \bar T)$ whenever $x\in B$, $\bar B\subset \Omega$, for some
$T=T(B)>0$. 
  Then for every $B_\rho(x_0)$, and $0<R<\rho$  
there is a $T\in (0, \bar T(B_R(x_0)))$ such that
$F_t(x)\in B_\rho(x_0)$ for a.e.\ $x\in B_R(x_0)$ and all $t\in [-T,T]$.
\end{lemma}

\begin{proof}
	Take a positive
	\[
	T<\frac{\rho-R}{\|V\|_{L^\infty(B_{\rho}(x_0))}}\wedge \bar T(B_R(x_0)).
	\]
	We have then that $F_t$ maps almost all of $B_R(x_0)$ into $B_{\rho}(x_0)$ for every $t\in [-T, T]$.
	We show this for an arbitrary $t\in [0, T]$, the remaining case being completely symmetric.
	In fact, otherwise there is an $x\in B_{\rho}(x_0)$ and $y\in C([0,\bar t];\R^d)$, $y(t):= F_t(x)$
	such that $y(t)\in B_{\rho}(x_0)$ for all $t\in [0,\bar t)$, $|y(\bar t)-x_0|=\rho$
	and
	$\bar t \leq T$. But	
	\[
	F_t(x)= x+\int_0^t V(x(\tau))\d \tau,
	\]
	so that 
	\[
	|F_{\bar t}(x)-x_0|\leq R + \bar t \|V\|_{L^\infty(B_{\rho}(x_0))}\leq  R + T \|V\|_{L^\infty(B_{\rho}(x_0))}< \rho,
	\]
	which is impossible because $|F^1_{\bar t}(x)-x_0|=|y(\bar t)-x_0|=\rho$ by assumption.
\end{proof}

\appendix

\section{Concentration on integral curves}

We prove here our main technical tool, namely, a criterion of when a measure 
over continuous curves is concentrated over
on integral curves of a given vector field $V$. Let us remark that it is based on a discrete formulation of the equation for integral curves. 

\begin{proposition}\label{prop:taylor-first-order}
	Let $p_0$, $p_1\in (1,+\infty)$ with $1/p_0+1/p_1 = 1/q \leq 1$.
	Assume that $V \in L^1\bra{[a,b]; L^{p_0}(\R^d)}$, $\DD V \in L^1\bra{[a,b]; L^{p_1}(\R^d)}$, where $a<0<b$ and $\eta$ be a finite measure on $C([a,b];\R^d)$ with $e_{t\#} \eta  \leq C\mathcal{L}^d$ for all $t\in [a,b]$. Then, $\eta$ is concentrated on integral curves of $V$, if and only if there exists a variation function $(\omega_{st})_{st \in [a,b]}$ with $[\omega] = 0$ 
	such that
	\begin{equation}\label{eq:taylor-first-order} \nor{ \theta_t - \theta_s - \bra{\int_s^t V_\tau \d \tau} (\theta_s) }_{L^q(C([a,b];\R^d),\eta)} \le \omega_{st} \quad \text{for every $s, t \in [a,b]$.}
	\end{equation}
	In such a case one can always choose
	\begin{equation}\label{eq:omega-1} \omega_{st} \lesssim C^{1/q} \bra{\int_s^t \nor{V_\tau}_{p_0} \d\tau } \bra{\int_s^t\nor{\DD V_\tau}_{p_1} \d\tau }.\end{equation}
\end{proposition}

\begin{remark}
	If \eqref{eq:taylor-first-order} holds, then we can always represent $\eta = F_{\#} (e_{0\#}\eta)$, where $F$ denotes the regular Lagrangian flow for $V$.
\end{remark}

\begin{proof}
	Assume first that $\eta$ is concentrated on integral curves of $V$. Then
	\begin{equation}\label{eq:bound-gamma-ts} \nor{ \theta_t - \theta_s}_{p_0} = \nor{ \int_s^t V_\tau(\theta_\tau) \d\tau }_{p_0} \le  C^{1/p_0}\int_s^t \nor{ V_\tau}_{p_0} \d\tau.
	\end{equation}
	It follows that
	\begin{equation}\label{eq:proof-taylor-first-order}\begin{split}  & \nor{ \int_s^t V_\tau(\theta_r) - \bra{\int_s^t V_\tau \d\tau} (\theta_s) }_{L^q(C([a,b];\R^d),\eta)} \\
	& \le \int_s^t \nor{ V_\tau(\theta_\tau) - V_\tau(\theta_s)}_{L^q(C([a,b];\R^d), \eta)} \d\tau  \\
	& \lesssim \int_s^t\nor{ |\theta_\tau - \theta_s|  \bra{|\DD V_\tau|^*(\theta_\tau) + |\DD V_\tau|^*(\theta_s)  } }_{L^q(C([a,b];\R^d), \eta)} \d\tau \qquad \text{by \eqref{eq:pontiwise-sobolev-star},}\\
	& \le  \int_s^t \nor{ \theta_\tau - \theta_s}_{L^{p_0}(C([a,b];\R^d),\eta)} \nor{ |\DD V_\tau|^*(\theta_\tau) + |\DD V_\tau|^*(\theta_s)}_{L^{p_1}(C([a,b];\R^d),\eta)}  \d\tau \\
	& \lesssim \int_s^t C^{1/p_1} \nor{ \DD V_\tau }_{p_1} C^{1/p_0}\int_s^\tau \nor{V_u}_{p_0} \d u \d\tau \quad \text{by \eqref{eq:bound-gamma-ts} and \eqref{eq:maximal},} \\
	& \le C^{1/q} \int_s^t  \nor{ \DD V_\tau }_{p_1} \d\tau  \int_s^t  \nor{ V_\tau}_{p_0} \d\tau\end{split}
	\end{equation}
	as claimed. Conversely, assume that \eqref{eq:taylor-first-order} holds. Then
	\begin{equation}\label{eq:estimate-zeroth-order} \nor{ \theta_t - \theta_s}_{L^{p_0}(\eta)} \le  \omega_{st} + \nor{ \int_s^t V_\tau(\theta_\tau) \d\tau }_{L^{p_0}(\eta)} \le \omega_{st} +  C^{1/p_0} \int_s^t \nor{ V_\tau}_{p_0} \d\tau.
	\end{equation}
	Applying the triangle inequality along any partition $\pi$ of $[s,t]$, we obtain that
	\[ \nor{ \theta_t - \theta_s}_{L^{p_0}(\eta)} \le \sum_{t_i \in \pi} \omega_{t_i t_{i+1}} +  C^{1/p_0} \int_s^t \nor{ V_\tau}_{p_0} \d\tau.\]
	Therefore, \eqref{eq:estimate-zeroth-order} self-improves to \eqref{eq:bound-gamma-ts} and we can argue then exactly as in \eqref{eq:proof-taylor-first-order}.
	It follows that
	\[  \nor{ \theta_t - \theta_s - \int_s^t V_\tau (\theta_\tau) \d\tau }_{L^q(C([a,b];\R^d), \eta)} \le  \omega_{st} +  C^{1/q} \int_s^t  \nor{ \DD V_\tau }_{p_1} \d\tau  \int_s^t  \nor{ V_\tau}_{p_0} \d\tau.\]
	This entails that $f(t) := \theta_t -\theta_0 - \int_0^t V_\tau(\theta_\tau) \d\tau$ has zero total variation, i.e., it must be $\eta$-a.e.\ constant (and null since $f(0) = 0$).
\end{proof}

As curious consequence of Proposition~\ref{prop:taylor-first-order}, we obtain the following quantitative stability estimate for regular Lagrangian flows, which is of interest itself though not used elsewhere in this paper. 

In particular it quantifies the qualitative statement that weak convergence of vector fields in time and strong convergence in space leads to convergence of the respective flows, see~e.g.\ \cite[remark 2.11]{crippa2008estimates} or \cite{gigli2020korevaar} for a Trotter-type formula.

To state it we write
\[ \Phi^\delta(x) := \log(1 + \delta^{-1} |x| )\]
for $x\in \R^d$, $\delta>0$,
and notice that, for every $x$, $y \in \R^d$,
\begin{equation} \label{eq:estimate-Phi-delta}\Phi^\delta(y) \le \Phi^\delta(x) + \frac{|y-x|}{\delta + |x| }.\end{equation}
Indeed,
\[ \begin{split} \log(1 + \delta^{-1} |y| ) & \le \log(1+ \delta^{-1} (|x|+|y|)) = \log\bra{ (1+ \delta^{-1}|x|) \bra{  1 +\frac{|y-x|}{\delta + |x| }} }\\
& \le \log\bra{ 1+ \delta^{-1}|x| } +  \frac{|y-x|}{\delta + |x| }.\end{split}\]

\begin{corollary}\label{cor:first-order}
For $i=1,2$ let 
\[
V^i \in L^1\bra{[a,b]; L^{p_0}(\R^d)}, \quad \DD V^i \in L^1\bra{[a,b]; L^{p_1}(\R^d)}, 
\]
with $p_0$, $p_1$, $a$, $b$ as in Proposition~\ref{prop:taylor-first-order},
$\eta^i$ be finite measures on $\Theta:=C([a,b];\R^d)$ with $e_{t\#} \eta^i  \leq C_i\mathcal{L}^d$ for all $t\in [a,b]$ 
concentrated on integral curves of $V^i$. Then, for every $\delta >0$ one has
\[  \begin{split}   \Big\|\sup_{i} \Phi^\delta( \theta^1_{t_i} - \theta^2_{t_i})  & 
\Big\|_{L^q(\Theta\times\Theta, \eta)}  \lesssim  \nor{  \Phi^\delta( \theta^1_{0} - \theta^2_{0} ) }_{L^q(\Theta\times\Theta, \eta)} +  \sum_i \nor{ \DD \int_{t_i}^{t_{i+1}} V^1_\tau \d \tau }_{p_1} \\
&  \quad +  \frac 1 \delta \sum_{i}  \nor{ \int_{t_i}^{t_{i+1}} (V^1_\tau - V^2_\tau) \d \tau  }_{p_0}
+ \frac 1 \delta\sum_i \left( \omega^1_{t_{i}t_{i+1}} + \omega^2_{t_i t_{i+1}}\right) ,
\end{split}\]
whenever $a = t_0 < t_1 \ldots < t_n \le b$,
where $\omega^i_{st}$ are as in~\eqref{eq:omega-1} (with $C_i$, $V^i$ instead of $C$, $V$), and $\eta$ is any coupling between $\eta^1$ and $\eta^2$ (i.e.\ a Borel measure over $\Theta\times\Theta$ with marginals $\eta^1$ and $\eta^2$).
\end{corollary}

\begin{proof}
Write 
\[
R_{st}^i :=  \theta^i_{t}- \theta^i_s - \left( \int_s^t V_\tau \d \tau\right) (\theta_s^i), 
\]
so that $\nor{R_{st}^i}_q \le  \omega^i _{st}$. We use \eqref{eq:estimate-Phi-delta} with $y= \theta^1_{t} - \theta^2_{t}$, $x = \theta^1_s - \theta^2_s$, obtaining
\[\begin{split}  \Phi^\delta ( \theta^1_{t} - \theta^2_{t} ) & \le \Phi^\delta ( \theta^1_{s} - \theta^2_{s} ) + \frac{ \abs{ \int_s^t V_\tau^1 \d \tau (\theta^1_s) - \int_s^t V^2_\tau \d \tau  (\theta^2_s) } + \abs{R^1_{st}} + \abs{ R^2_{st} } }{ \delta + | \theta^1_{s} - \theta^2_{s} | }\\
& \le \Phi^\delta ( \theta^1_{s} - \theta^2_{s} ) +    \frac{ \abs{ \int_s^t  V_\tau^1\d \tau (\theta^1_s) - \int_s^t V_\tau^1 \d \tau  (\theta^2_s)} }{| \theta^1_{s} - \theta^2_{s} |}  \\
& \quad +  \delta^{-1} \bra{ \abs{ \int_s^t ( V_\tau^1 - V_\tau^2) \d \tau (\theta^2_s) } + \abs{R^1_{st}} + \abs{ R^2_{st} } }.\end{split}\]
To conclude, it is sufficient to choose $s := t_{i}$, $t:=t_{i+1}$, proceed recursively, and finally take the $L^q$ norms (recall that we are considering finite measures, so that $L^p$ spaces are nested).
\end{proof}

\section{Auxiliary lemmata}

The lemma below provides a useful version of the Lebesgue dominated convergence theorem.

\begin{lemma}\label{lm_weak_meas_conv1}
	If $\{f_k\}\subset L^1(E,\mu)$ is a bounded and equintegrable sequence of nonnegative functions, and $g_k\to 0$ in measure $\mu$ with $0\leq g_k\leq C$
	$\mu$-a.e., then
	\[
	\int_E f_kg_k \d\mu \to 0. 
	\]
\end{lemma}

\begin{proof}
	For every $\varepsilon>0$ let $E_{k,\varepsilon}:= \{x\in E\colon |g_k(x)|>\varepsilon\}$.
	Then
	\[
	\int_E f_kg_k \d\mu  = \int_{E_{k,\varepsilon}} f_kg_k \d\mu  + 
	\int_{E\setminus E_{k,\varepsilon}} f_kg_k \d\mu  \leq C \int_{E_{k,\varepsilon}} f_k\d\mu + \varepsilon \|f_k\|_{L^1(E,\mu)},
	\]
	the first term vanishing as $k\to\infty$ because 
$\lim_k \mu(E_{k,\varepsilon})=0$ and in view of
	 equiintegrability of $f_k$.
	Since $\varepsilon>0$ is arbitrary, one gets the claim.
\end{proof}

We will use also another lemma very similar to the previous one, which we formulate for purely notational convenience for
one-parameter families of functions rather than for sequences. 

\begin{lemma}\label{lm_domconvmeas1}
	Let $\cur{\nu_s}_{s \in [a,b]}$ be a tight family of absolutely continuous 
	Borel measures $\nu_s = f_s\mathcal{L}^d$ such that 
	\[ \sup_{s \in [a,b]} \nor{ f_s}_ \infty  < \infty.\]
	Let $\cur{g_\delta}_{\delta \in [0,1]}$ be a family of non-negative Borel functions with $0\leq g_\delta \leq C$ for all
	$\delta$ and 
	$\lim_{\delta \to 0} g_{\delta} = 0$ in measure. 
	Then
	\[\lim_{\delta \to 0} \sup_{s \in [a,b]} \int g_{\delta} \d \nu_s = 0.\]
\end{lemma}

\begin{proof}
	For any $\eps>0$, let $K \subseteq \R^d$ be compact with $\sup_{s \in [a,b]} \nu_s(K^c)< \eps$. Then, 
	\[\begin{split} \int_{\R^d} g_{\delta} \d \nu_s &  = \int_K g_{\delta}f_s\d x + \int_{K^c} g_\delta f_s \d x
	\le C \int_K g_{\delta}\d x + \sup_{\delta \in [0,1]} \nor{g_\delta}_\infty \nu_s(K^c) \\
	& \le  C \bra{ \int_K g_{\delta}\d x  + \eps}.
	\end{split}\]
	By Lebesgue dominated convergence theorem one has $\lim_{\delta \to 0}\int_K g_{\delta} \d x = 0$, and hence
	\[ \limsup_{\delta \to 0} \sup_{s \in [a,b]} \int g_{\delta} \d \nu_s  \le C \eps,\]
	and we conclude since $\eps>0$ is arbitrary small.
\end{proof}

The following lemma gives a chain rule for Sobolev functions along ``almost every'' integral curve of an ODE.

\begin{lemma}\label{lm_chainrule_sobolev_curves}
	Let $p, q \ge 1$ with $\frac{1}{p} + \frac 1 q \le 1$, let $f\in W^{1,p}(\R^d)$, $V \in L^1([a,b]; L^{q}(\R^d))$ and $\eta$ be a finite measure on $C([a,b]; \R^d)$ with bounded compression and concentrated on integral curves of $V$. Then, for any a.e.\ representative of $f$ one has that for $\eta$-a.e.\ $\theta$ the curve $t \mapsto f(\theta_t)$ is absolutely continuous and
	\[ \partial_t f(\theta_t) = V(\theta_t) \nabla f(\theta_t) \quad \text{for a.e.~$t \in [a,b]$.}\]
\end{lemma}

\begin{proof}
	We approximate $f$ with a fast converging sequence  $(f_n)_{ n\ge 1} \subseteq  C^1(\R^d)$, i.e., such that
	\[ \sum_{n =1 }^\infty \nor{ f - f_n}_{p} + \nor{ \nabla f - \nabla f_n}_p < \infty.\]
	The thesis clearly holds for every $f_n$ in place of $f$. We have
	\[ \nor{ f_n \circ e_t - f \circ e_t}_{L^p(\eta)} \lesssim \nor{ f_n -f}_p \quad  \text{for every $t \in [a,b]$,}\]
	so that for every $t \in [a,b]$, $\eta$-a.e., $\lim_{n \to \infty} f_n(\theta_t) = f(\theta_t)$. In fact the limit holds $\eta$-a.e.\ and uniformly with respect to $t \in [a,b]$. To show this, we notice that
	\[\begin{split}
	\nor{ \sup_{t \in [a,b]} (f_{n+1}- f_n) \circ e_t}_{L^(\eta)} & \le \nor{ (f_{n+1} - f_n)\circ e_0}_{L^1(\eta)}  \\
	&\quad +\nor{ \int_a^b |\nabla f_{n+1} - \nabla f_{n+1}| | V_\tau| \circ e_\tau \d \tau}_{L^1(\eta)} \\
	& \lesssim \nor{ f_{n+1} - f_n}_{p}\\
	& \quad + \int_a^b \nor{ \nabla f_{n+1} - \nabla f_n}_{p} \nor{V_\tau}_q \d \tau.
	\end{split}\]
	The last quantity is summable with respect to $n \ge 1$. It follows that $\eta$-a.e.\ the sequence $(f_n(\theta_s)_{s \in [a,b]}$ is Cauchy, hence convergent in $C([a,b])$, towards $(f(\theta_s))_{s \in [a,b]}$, that is in particular $\eta$-a.e. continuous.
	
	To argue that it is indeed absolutely continuous, we use the bound
	so that, $\eta$-a.e., for every $t \in [a,b]$,
	\[\lim_{n \to \infty} \int_0^t \nabla f_n (\theta_\tau) V_\tau(\theta_\tau) \d \tau = \int_s^t \nabla f (\theta_\tau) V_\tau(\theta_\tau) \d \tau.
	\]
	To conclude, we simply pass to the limit in the $\eta$-a.e.\ identity
	\[ f_n(\theta_t) - f_n(\theta_0) = \int_a^t \nabla f_n (\theta_\tau) V_\tau(\theta_\tau) \d \tau.\]
\end{proof}

\bibliographystyle{plain}
\bibliography{frobenius_final}
\end{document}